\documentclass[a4paper,11pt,oneside,nohead]{amsart}

\usepackage{mathabx, stmaryrd}
\usepackage[english]{babel}
\usepackage[utf8x]{inputenc}    
\usepackage{xcolor}            
\usepackage{xspace}
\usepackage[pdftex]{graphicx}
\usepackage{epstopdf}
\usepackage{mathrsfs}
\usepackage[lofdepth,lotdepth]{subfig}

\usepackage{mathabx, stmaryrd}
\numberwithin{equation}{section}

\newtheorem*{theorem-non}{Theorem}

\usepackage{graphicx, graphics}          
\usepackage{dsfont }
\usepackage{caption}           
\usepackage{url}

\usepackage[a4paper,scale={0.72,0.74},marginratio={1:1},footskip=7mm,headsep=10mm]{geometry}

\usepackage{hyperref}
\hypersetup{					
	plainpages=false,			
	colorlinks=true,			
	pdfborder={0 0 0},			
	breaklinks=true,			
	bookmarksnumbered=true,		%
	bookmarksopen=true			%
}

\newcommand{\sumtwo}[2]{\sum_{\substack{#1 \\ #2}}}
\newcommand{\suptwo}[2]{\sup_{\substack{#1 \\ #2}}}

\newcommand{\ind}{\mathds{1}}

\newcommand{\bP}{\mathbf{P}}
\newcommand{\bE}{\mathbf{E}}
\newcommand{\bbP}{\mathbb{P}}
\newcommand{\bbE}{\mathbb{E}}
\newcommand{\bbR}{\mathbb{R}}
\newcommand{\bbZ}{\mathbb{Z}}

\newcommand{\ent}{\mathrm{Ent}}
\newcommand{\len}{\ell}

\newcommand{\cP}{{\ensuremath{\mathcal P}} }
\newcommand{\cE}{{\ensuremath{\mathcal E}} }
\newcommand{\cH}{{\ensuremath{\mathcal H}} }
\newcommand{\cC}{{\ensuremath{\mathcal C}} }
\newcommand{\cN}{{\ensuremath{\mathcal N}} }
\newcommand{\cL}{{\ensuremath{\mathcal L}} }

\newcommand{\cT}{{\ensuremath{\mathcal T}} }

\newcommand{\cB}{{\ensuremath{\mathcal B}} }

\newcommand{\dd}{\textrm{d}}

\renewcommand{\epsilon}{\varepsilon}
\renewcommand{\phi}{\varphi}

\newtheorem{theorem}{Theorem}[section]
\newtheorem{conj}[theorem]{Conjecture}

\newtheorem{proposition}[theorem]{Proposition}

\newtheorem{lemma}[theorem]{Lemma}

\theoremstyle{definition}

\newtheorem{remark}[theorem]{Remark}

\newcommand{\ga}{\alpha}
\newcommand{\gb}{\beta}

\newcommand{\gep}{\varepsilon}       

\newcommand{\gD}{\Delta}

\newcommand{\gU}{\Upsilon}

\newcommand{\sE}{\mathscr{E}}

\newcommand{\sL}{\mathscr{L}}

\title[Local and global constraints in LPP]{Beyond Hammersley's Last-Passage Percolation: a discussion on possible local and global constraints}

\author[Q. Berger]{Quentin Berger}
\address{Sorbonne Universit\'e, LPSM,
Campus Pierre et Marie Curie, case courrier 158,
4 place Jussieu, 75252 Paris Cedex 5, France}
\email{quentin.berger@sorbonne-universite.fr}

\author[N. Torri]{Niccol\`o Torri}
\address{Sorbonne Universit\'e, LPSM,
Campus Pierre et Marie Curie, case courrier 158,
4 place Jussieu, 75252 Paris Cedex 5, France}
\email{niccolo.torri@sorbonne-universite.fr}

\date{}

\begin{document}

\begin{abstract}
Hammersley's Last-Passage Percolation (LPP), also known as Ulam's problem, is a well-studied model that can be described as follows: consider $m$ points chosen uniformly and independently in $[0,1]^2$, then what is the maximal number $\mathcal{L}_m$ of points that can be collected by an up-right path? We introduce here a generalization of this standard LPP, in order to allow for more general constraints than the up-right condition (a $1$-Lipschitz condition after rotation by $45^{\circ}$).
We focus more specifically on two cases: (i) when the constraint is a $\gamma$-H\"older (local) condition, we call it H-LPP; (ii) when the constraint is a path-entropy (global) condition, we call it E-LPP. 
These generalizations also allows us to deal with \textit{non-directed} LPP.
We develop motivations for directed and non-directed constrained LPP, and we give the correct order of $\mathcal{L}_m$ in a general manner.
\\[0.05cm]
\end{abstract}

\thanks{N. Torri was supported by a public grant as part of the “Investissements d’Avenir” program (ANR-11-LABX-0020-01 and  ANR-10-LABX-0098). 
Q. Berger acknowledges the support of grant ANR-17-CE40-0032-02. Both authors also acknowledge the support of PEPS grant from CNRS, which led to the development of this work.
The authors are also most grateful to N. Zygouras for several discussions on the subject, and in particular for discussing the application presented in Section~\ref{sec:appli1}.
}

\maketitle


\section{Introduction}
\label{sec:intro}

In this introduction, we recall the original Hammersley LPP problem, and we show how to generalize this process by enlarging the set of paths allowed to collect points, by changing the \emph{increasing} constraint, to a more general \emph{compatibility} condition.  
We point out that the condition in the Hammersley's LPP is \emph{local}, the constraint depending only on two consecutive points. Conversely, a \emph{global} condition is a constraint that takes in account the whole path trajectory that collects points. 

In Section \ref{sec:LPP}, we introduce some specific constraints of interest (local and global) in the directed setting and we derive the correct order for the corresponding LPP problems. In Section~\ref{sec:nondir}, we consider a  non-directed LPP and we also derive its correct order.
In Section~\ref{sec:applications}, we present some contexts where our generalized  LPP can be useful---it has already proven useful in \cite{cf:BT_HT}.
We conclude the paper by presenting some simulations, and discuss some conjectures, see Appendix~\ref{app}.

\subsection{Hammersley's Last Passage Percolation}\label{secH}

Let us take $m$ points independently as uniform random variables in the square $[0,1]^2$, and denote the coordinates of these points $Z_1:=(t_1,x_1), Z_2:=(t_2,x_2),$ etc...
We say that a sequence $(z_{i_\ell})_{1\leq \ell\leq k}$ is \emph{increasing} if
$t_{i_\ell} >t_{i_{\ell-1}}$ and $x_{i_{\ell}}> x_{i_{\ell-1}}$ for any $1\leq \ell\leq k$ (we set by convention $i_0=0$ and $z_0=(0,0)$).

Then, the question is to study the length of the longest increasing sequence among the $m$ points which is equivalent to the length of the longest increasing subsequence of a random (uniform) permutation of length $m$.
We denote:
\[\cL_m  := \sup \big\{ k\ ; \exists\, (i_1,\ldots,i_k) \ s.t.\ (Z_{i_\ell})_{1\leq \ell\leq k} \text{ is increasing}\big\} \, .\]

Using sub-additive techniques, Hammersley  \cite {Ha72} first proved that $m^{-1/2} \cL_m$ converges a.s. and in $L^1$ to some constant, that was believed to be $2$. Further works then proven that the constant was indeed $2$ \cite{LS77,VK77}
. Moreover, and quite remarkably, this model has been shown to be exactly solvable by Baik, Deift and Johansson \cite{BDJ99}, and they identified the fluctuations of $\cL_m$ around $2\sqrt{m}$, showing that the model is in the so-called KPZ universality class. More precisely, in \cite{BDJ99}
the authors showed the following result.
\begin{theorem-non}
\label{thm:LPP}
The recentered and renormalized quantity $m^{-1/6} (\cL_m - 2\sqrt{m})$ converges in distribution
to the Tracy-Widom GUE distribution.
\end{theorem-non}
Moreover, Johansson \cite{Jo99} proved that the typical transversal fluctuations of a path collecting
the maximal number of points is of order $m^{-1/6}$.
(In \cite{Jo99} Johansson actually considers up-right paths going from $(0,0)$ to $(N,N)$ in a Poisson Point process of intensity~$1$: he shows that the typical transversal fluctuations away from the diagonal of a path collecting the maximal number of points is of order $N^{2/3}$.)
We also mention that in \cite{DZ95}, the case when the points are not chosen uniformly in $[0,1]^2$ but have some given density $p(x,y)$ has also been solved: the limiting constant $\lim_{n\to\infty} \cL_m/\sqrt{m}$ and the limiting curve  are identified.

\subsection{General definition of  path-constrained Last Passage Percolation}

We now perform a $45$ degree clockwise rotation, and generalize Hammersley's LPP by introducing a general constraint on paths (that can be either \emph{local} or \emph{global}): we introduce it via a notion of \emph{compatibility} of the points that can be collected.
We need three ingredients:
\begin{itemize}
\item a domain $\Lambda$;
\item a (finite or countable) random set of points $\gU \subset \Lambda$, whose elements are denoted by $Z_i = (t_i,x_i)$, its law is denoted $\mathbb{P}$;
\item a compatibility condition, \textit{i.e.}\ a set $\cC$ of compatible subsets of $\Lambda$. 
\end{itemize}
Then, we define the \emph{$\cC$-compatible Last-Passage Percolation} as the maximal number of $\cC$-compatible points in $\gU$, that is 
\begin{equation}
\cL_{\gU}^{(\cC)}=\cL_{\gU}^{(\cC)}(\Lambda) := \sup \Big\{  |\Delta|\,  ; \,  \Delta \subset \gU , \Delta \in \cC \Big\}\, .
\end{equation}

\begin{remark}\rm
This fits the definition of Hammersley's LPP as defined above: the compatibility set $\cC$ being the set of all increasing subsets of $[0,1]^2$.
After a rotation by 45$^\circ$, we end up with the domain $\Lambda:=\big\{(x,y), 0<x<\sqrt{2}, |y| \leq \min(1,1-t)\big\}$, and $\gU=\gU_m$ a set of $m$ independent uniform random variables in $\Lambda$. The compatibility set is then taken to be (with the convention $(t_0,x_0)=(0,0)$)
\[\cC=\bigcup_{k\ge 0} \Big\{  \Delta = \{(t_i,x_i)\}_{1\leq i\leq k} \ ; 0<t_1<\cdots< t_k <\sqrt{2} ,\  \frac{|x_i-x_{i-1}|}{|t_i-t_{i-1}|} \leq 1 \text{ for all } 1\le i\le k  \Big\} \, , \]
which corresponds to sets of points that can be collected via a $1$-Lipschitz function.
\end{remark}

Now, there are at least two reasonable ways of defining the compatibility condition: 
(i) by replacing the Lipschitz condition by a H\"older constraint;
(ii) by considering a path-entropy constraint (a global constraint on the path), that also allows to deal with non-directed paths. 
We restrict ourselves to the case of the dimension $d=2$ for the simplicity of the exposition, but it can easily be extended to the case of higher dimensions.

{
Several other constraints can be (and have been) considered, and let us mention a few. For instance the constraint that the path is convex has been studied in \cite{AB09}, and is related to the question of counting the number of lattice convex shapes, see \cite{Bar95,Sin94,Ver94} and more recently in \cite{BE16}. The question of pattern-avoiding permutation has also gained some interest recently, see in particular \cite{HRS17,MP16,MP14}.
}

\section{Directed LPP: H\"older and entropy constraints}
\label{sec:LPP}

In this section, we consider directed paths. We work with a domain $\Lambda_{t,x} = [0,t]\times [-x,x]$, for some (fixed) $t,x>0$. Then, we let $m$ independent r.v.\ uniform in $\Lambda_{t,x}$ form the set~$\gU_m$. We will use  $\cL_m$ as a short notation for $\cL_{\gU_m}$.
Moreover, we say that a set  $\Delta = \{(t_i,x_i)\}_{1\leq i\leq k}\subset \bbR_+ \times \bbR$ is \emph{directed} if
$0<t_1<\cdots<t_k$. 
We deal first with the (local) H\"older constraint, before we turn to the (global) Entropy constraint.

\subsection{Local H\"older constraint}
\label{sec:HLPP}

For any $\gamma \ge 0$, we can define the $\gamma$-H\"older norm of a finite set $\Delta = (t_i,x_i)_{1\le i\le k}$ (in which the points are ordered $t_1<\cdots <t_k$, with the convention $(t_0,x_0)=(0,0)$)
\begin{equation}
\label{def:Holder}
\mathrm{H}_{\gamma} (\Delta) := \sup_{1\le i\le k}  
\   \frac{|x_i-x_{i-1}|}{|t_i-t_{i-1}|^{\gamma}}  \, .
\end{equation}
Notice that this is not the $\gamma$-H\"older norm  of the linear interpolation of the points, since \eqref{def:Holder} only considers consecutive points: one can think of this quantity as a \textit{local} $\gamma$-H\"older norm. In particular,  the case $\gamma>1$ is not trivial here, and the case $\gamma=0$ is also of interest.
Then, for some fixed $A>0$, we define a compatibility set 
\begin{equation}
\cH_{\gamma}^{A} := 
\big\{ \Delta \subset \bbR_+ \times \bbR\, ;\, \Delta \text{ directed}, \, \mathrm{H}_{\gamma} (\Delta) \le A  \big\} \, 
.\end{equation}

The $\gamma$-H\"older Last Passage Percolation, abbreviated as H$_{\gamma}$-LPP, is defined as
\begin{equation}
\label{def:HLPP}
\mathcal{L}_m^{(\cH_{\gamma}^{A})} ( \Lambda_{t,x}) := \sup \Big\{  |\Delta| \, ; \, \Delta \subset \gU_m , \Delta \in  \cH_{\gamma}^{A} \Big\} \, .
\end{equation}
We prove the following result.

\begin{theorem}
\label{thm:HLPP}
There are constants $c_1,c_2$ (depending only on $\gamma$, during the course of the proof one finds that $c_1\le c(1+\gamma)^{-1/2} $) such that for any $t,x$ and $B$, for any $1\le k\le m$
\begin{align}
\label{HLPP-upper}
\bbP\Big( \mathcal{L}_m^{(\cH_{\gamma}^{A})} (\Lambda_{t,x}) \ge k  \Big) & \le  \Big( \frac{c_1 A t^{\gamma} m }{x k^{1+\gamma}} \Big)^k \, ,\\
\bbP\Big( \mathcal{L}_m^{(\cH_{\gamma}^{A})} (\Lambda_{t,x}) \le  k \Big) & \le  \exp\bigg\{ c_2 k \Big( 1- c_2 \Big( \frac{At^{\gamma}}{ x k^{\gamma} } \wedge 1   \Big)  \frac{m}{k} \Big) \bigg\} \, .
\label{HLPP-lower}
\end{align}
As a consequence, there is some $C>0$ such that for any fixed $t,x,\gamma,A$, $\bbP$-a.s.\ there is some $m_0$ such that
\[\frac{1}{C}\le  \frac{\mathcal{L}_m^{(\cH_{\gamma}^{A})} (\Lambda_{t,x}) }{ (At^{\gamma}/x )^{1/(1+\gamma)} m^{1/(1+\gamma)} } \le C \quad \text{ for all } m\ge m_0 \, .\]
\end{theorem}
We stress that the constants in \eqref{HLPP-upper}-\eqref{HLPP-lower} are uniform in the parameters $m,A,t,x$: the results are still valid when considering the situation when $A,t,x \to \infty$ as $m\to\infty$, which is useful for some applications, see \cite{cf:BT_HT}. 

We have that  $\mathcal{L}_m^{(\cH_{\gamma}^{A})} =\cL_m^{(\cH_{\gamma}^{A})}(\Lambda_{t,x})$ is of order $m^{\kappa}$, with $\kappa=1/(1+\gamma)$. Then, it is very natural to expect that $\mathcal{L}_m^{(\cH_{\gamma}^{A})}   / m^{\kappa}$  converges a.s.\ to a constant as $m\to\infty$: we discuss this convergence in Section~\ref{sec:poisson}, see in particular Remark~\ref{rem:converge}. The value of the constant is discussed in Appendix~\ref{app}.

%

\begin{remark}\rm 
One could naturally generalize H\"older LPP to a \emph{cone-shaped} LPP: define a  region $\mathcal R =\{ (t,x) \in  \bbR_+ \times \bbR, f_2(t) \le x\le f_1(t)\}$, with $f_1\le f_2$ two functions $\bbR_+ \to \bbR$, and let the compatibility condition for $\Delta$ be that for any $(t_{i-1},x_{i-1}),(t_{i},x_i) \in \Delta$ we have $(t_i-t_{i-1},x_i-x_{i-1}) \in \mathcal R$ (\textit{i.e.}\ the next point in $\Delta$ has to be in the cone-shaped region $\mathcal R$ from the previous point).
In this framework, H$_{\gamma}$-LPP is simply the \emph{cone-shaped} LPP with $\mathcal R = \{ (t,x), -t^{\gamma} \le x\le t^{\gamma}\}$, and one could easily adapt the proof of Theorem~\ref{thm:HLPP}: the key quantity is $V(u) = \int_0^u |f_1-f_2|(v) dv$, the area of $\mathcal R$ close to the origin, and  one finds that $\cL_m$ is of the order of $V^{-1}(1/m)$ (recovering the $m^{1/(1+\gamma)}$ in the H\"older case).
\end{remark}

\subsection{Global Entropy constraint}
\label{sec:ELPP}

Another type of constraint that is natural to consider is a \emph{global} constraint: we talk about an entropy constraint, since it arises naturally when considering random walk paths (the entropy being a measure of the non-likelihood of a path).
This is a generalization of the study initiated in~\cite{cf:BT_ELPP}, which was motivated by applications to directed polymer in random heavy-tail environment and helped answer Conjecture~1.7 in \cite{cf:DZ}.
For any $a\ge b\ge 0$, $a>0$, we define the $(a,b)$-Entropy of a set $\Delta = (t_i,x_i)_{1\le i\le k}$ (again, the points are ordered $t_1<\cdots <t_k$, and we use the convention $(t_0,x_0)=(0,0)$)
\begin{equation}
\label{def:ent}
\ent_{a,b} (\Delta) := \sum_{i=1}^k \frac{|x_i-x_{i-1}|^a}{ |t_i-t_{i-1}|^b}\, .
\end{equation}
In particular, we are interested in two special subcases. First, when $b>0$ and $a=b+1$: in that case, we can generalize the notion of entropy to continuous paths $s:[0,t]\to \bbR$, by $\ent_{b}(s) = \int_0^t |s'(u)|^b du$, corresponding to  the $L^b$ norm of $s'$ (it is related to the $(1,b)$-Sobolev norm of $s$) and the entropy of a set $\Delta$ corresponds to the entropy of the linear interpolation of $\Delta$.
Second, when $b=0$: then the entropy can be generalized to non-necessarily continuous paths $s:[0,t]\to \bbR$, by $\ent_a(s) = \sup \big\{ \sum_i  |s(t_i)-s(t_{i-1}) |^a \big\}$, the supremum being over all finite subdivisions $t_1<\cdots<t_k$ of $[0,t]$. This corresponds to the ``$a$-variation'' norm of $s$ (the total variation for $a=1$, and the quadratic variation for $a=2$).

\smallskip

For fixed $B>0$, we define a compatibility set 
\begin{equation}
\cE_{a,b}^{B} := 
\Big\{ \Delta \subset \bbR_+ \times \bbR \, ;\, \Delta \text{ directed}, \,  \ent_{a,b} (\Delta) \le B  \Big\} \, ,
\end{equation}
so that a set of points $\Delta$ is compatible if it can be collected by a path with entropy smaller than $B$.
The path-Entropy constrained LPP, abbreviated as E-LPP, is then defined as
\begin{equation}
\mathcal{L}_m^{(\cE_{a,b}^{B})} (\Lambda_{t,x}) := \sup \Big\{  |\Delta| \, ; \, \Delta \subset \gU_m , \Delta \in  \cE_{a,b}^{B} \Big\} \, .
\end{equation}
We prove the following result.

\begin{theorem}
\label{thm:ELPP}
There are constants $c_3,c_4$ (depending only on $a,b$) such that for any $t,x$ and any $B$, for any $1\le k\le m$
\begin{align}
\label{ELPP-upper}
\bbP\Big( \mathcal{L}_m^{(\cE_{a,b}^{B})} (\Lambda_{t,x}) \ge  k\Big) & \le  \Big( \frac{c_3 ( B  t^b/x^a)^{1/a} m }{k^{(a+b+1)/a}} \Big)^{ k}\, ,\\
\bbP\Big( \mathcal{L}_m^{(\cE_{a,b}^{B})} (\Lambda_{t,x}) \le  k\Big) & \le  \exp\bigg\{ c_4 k \Big( 1- c_4 \Big( \frac{ (B t^b/x^a)^{1/a} }{ k^{(a+b)/a} }\wedge 1   \Big)  \frac{m}{k} \Big) \bigg\}\, .
\label{ELPP-lower}
\end{align}
As a consequence, there is a constant $C$ such that for any fixed $t,x,a,b,B$, $\bbP$-a.s.\ there is some $m_0$ such that
\[ \frac{1}{C}\le  \frac{\mathcal{L}_m^{(\cE_{a,b}^{B})} (\Lambda_{t,x}) }{ ( B  t^b/x^a)^{1/(a+b+1)}  m^{a/(a+b+1)} } \le C \quad \text{ for all } m\ge m_0 \, .\]
\end{theorem}
Again, the constants are uniform in the different parameters (and explicit, see the proof of Theorem~\ref{thm:ELPP}), and this fact could reveal to be useful, in particular for the problem developed in Section~\ref{sec:appli1}.

Also here, $\cL_m^{(\cE_{a,b}^{B})}=\cL_m^{(\cE_{a,b}^{B})} (\Lambda_{t,x})$ is of order $m^{\kappa}$ with $\kappa=a/(a+b+1)$, and it is natural to expect that $\cL_m^{(\cE_{a,b}^{B})}/ m^{\kappa}$  converges a.s.\ to a constant as $m\to\infty$. This convergence is discussed in Section~\ref{sec:poisson}, and the value of the constant in Appendix~\ref{app}.
Notice that in the case where $a=b+1$ (which is one of the most natural, since it arises from LDP of random walks, see Remark~\ref{rem:pathentropy}), we find $\kappa=1/2$, exactly as in the case of a Lipschitz constraint. In the case $b=0$, we find $\kappa=a/(a+1)$ so $\kappa=1/2$ when $a=1$ (total variation case) and $\kappa=2/3$ when $a=2$ (quadratic variation case).

\begin{remark}\rm
\label{rem:pathentropy}
The entropy of a set~$\gD$ as defined in \eqref{def:ent} appears naturally when considering large deviations for random walks: consider $S$  a symmetric random walk with unbounded jumps, with stretch exponential tail $\bP(S_1=x) \stackrel{x\to\infty}\sim e^{- |x|^{\nu}}$, for some $\nu>0$ (the case of the usual simple random walk corresponds to taking $\nu=\infty$). Then, when considering the probability that a point $(n,x_n)$ (with $n\to\infty$, $x_n \gg \sqrt{n}$) is visited (or collected) by the simple random walk path, we realize that
\begin{equation}
\label{eq:LDP}
-\log \bP(S_n = x_n) \stackrel{n\to\infty}{\sim}
\begin{cases}
n I(x_n/n) &\quad  \text{if } \nu>1, \text{ or }  \nu\in(0,1)\text{ and } x_n \ll n^{1/(2-\nu)} \, ,\\
J(x_n) & \quad \text{if } \nu\in(0,1)\text{ and } x_n \gg n^{1/(2-\nu)},
\end{cases}
\end{equation}
with some LDP rate functions $I(\cdot),J(\cdot)$. More specifically, we have $I(x)\sim x^2/2$ as $x\to 0$ (moderate deviation regime, see \cite{Cramer} for the standard Cram\'er case, \cite{Nag69} for the case $\nu\in(0,1)$), $I(x) = x^{\nu}$ as $x\to \infty$ (super-large deviation, see \cite[Thm.~2.1]{Nag79}), and $J(x)= x^{\nu}$ (one-jump deviation, see \cite[Thm.~2.1]{Nag79}). 
Hence, the entropy defined in \eqref{def:ent} is the natural scaling limit of the $\log$-probability that a random walk path visits a given set of points. We chose the specific form \eqref{def:ent} instead of using general LDP rate functions $I(\cdot),J(\cdot)$ because: (i) we are able to perform computations with this formula, (ii) we can usually bound the rate function $c|x|^a\le I(x)\le c' |x|^a$ for some $a>0$. In \eqref{eq:LDP}, we therefore have: in the first part $a=2,b=1$ if $x_n/n\to 0$ or $a=\nu, b=\nu-1$ ($\nu>1$) if $x_n/n\to\infty$; in the second part, $a=\nu,b=0$.
However we keep the parameters $a,b$ in the definition \eqref{def:ent}, to be able to deal with all these cases at once.
\end{remark}


\subsection{Poissonian (point-to-point) version of path-constrained LPP}
\label{sec:poisson}

Similarly to the standard LPP, we can define a Poissonian (point-to-point) version of the path constrained LPP, reproducing the idea of Hammersley \cite{Ha72} to prove the convergence of $\cL_m/\sqrt{m}$. 

For any $\lambda>0$, let $\gU_{\lambda}$ be a Poisson point process of intensity $\lambda$ on $\bbR^2$, and we define the point-to-point version of constrained LPPs. Let us consider $z=(x,y)\in \bbR^2$. For a given set $\Delta \subset \mathbb R \times (0,y)$, we set $\Delta^{(z)} = \Delta\cup \{z\}$ so that it extends $\Delta$ to make it end at~$z$.
We consider the domain $\Lambda_t=[0,t]\times \bbR$ for $t>0$, and for any $A>0$, $B>0$, we define the point-to-point constrained LPP:

\begin{align*}
\cL_{\gU_\lambda}^{(\cH_{\gamma}^{A})}(t)= \cL_{\lambda}^{(\cH_{\gamma}^{A})}(t) &:= \sup\Big\{ |\Delta| ; \Delta \subset \gU_{\lambda} \cap \Lambda_t  , \Delta \text{ directed},  \mathrm{H}_{\gamma}(\Delta^{(t,0)}) \le A \Big\} \, 
,\\
\cL_{\gU_\lambda}^{(\cE_{a,b}^{B})}(t) = \cL_{\lambda}^{(\cE_{a,b}^{B})}(t) &:= \sup\Big\{ |\Delta| ; \Delta \subset \gU_{\lambda} \cap \Lambda_t , \Delta \text{ directed}, \ent_{a,b}(\Delta^{(t,0)}) \le B t \Big\}  \, .
\end{align*}
Let us note that the entropy constraint grows linearly in $t$.
We realize that in both cases, $\big(\cL_{\lambda}^{(\cC)}(n)\big)_{n\ge 1}$ forms a super-additive ergodic sequence, in the sense that
\begin{equation}
\label{eq:subadd}
\cL_{\gU_\lambda}^{(\cC)}(n+\ell ) \ge \cL_{\gU_{\lambda}}^{(\cC)}(n)  + \cL_{\theta^n \gU_\lambda}^{(\cC)}(\ell) \, ,
\end{equation}
where $\theta^n$ is the  translation operator: $\theta^n(t,x) = (t+n,x)$.
The super-additivity comes from the fact that the concatenation of two sets have: (i) a H$_{\gamma}$ norm equal to the maximum of the H$_{\gamma}$ norms of the two sets; (ii) an entropy equal to the sum of the entropies of the two sets.
Therefore, Kingman's sub-additive ergodic theorem \cite{King} implies the existence of the limit 
$\lim_{n\to\infty} n^{-1} \cL_{\gU_\lambda}^{(\cC)}(n)$. In the following result we show that the limit is finite (and can be taken along the real line). 
\begin{proposition}\label{prop:subadditive}
For any $\lambda>0$, the limits
\begin{equation}
\label{subadditive}
\mathtt{C}_{\lambda,A}^{\rm H}= \lim_{t\to\infty}  \frac{1}{t} \cL_{\lambda}^{(\cH_{\gamma}^{A})}(t) , \qquad  
\mathtt{C}_{\lambda,B}^{\rm E} = \lim_{t\to\infty}  \frac{1}{t} \cL_{\lambda}^{(\cE_{a,b}^{B})}(t)
\end{equation}
exist a.s.\ and in $L^1$, and are finite, constant $\bbP$-a.s. 
Moreover the constants $\mathtt{C}_{\lambda, A}^{\rm H}$ and $\mathtt{C}_{\lambda,B}^{\rm E}$ satisfy the following scaling relations
\begin{equation}
\label{eq:scalingC}
\mathtt{C}_{\lambda, A}^{\rm H}=  
(\lambda A)^{\frac{1}{1+\gamma}}\, \mathtt{C}^{\rm H}_{1,1} \ ; \  \mathtt{C}_{\lambda, B}^{\rm E}(u)= (\lambda B^{1/a})^{\frac{a}{a+b+1}}\,  \mathtt{C}^{\rm E}_{1,1} .
\end{equation}
\end{proposition}
We refer to Appendix~\ref{app} for a discussion on the value of the constants.

\begin{proof} 
We have already noted that the limit along the integers $n\to\infty$ exists. We can extend the limit along the real line $t\to\infty$, using that $ \cL_{\lambda}^{(\cC)}(\lfloor t \rfloor)\le  \cL_\lambda^{(\cC)}(t) \le \cL_\lambda^{(\cC)}(\lceil t \rceil) $.

We now prove that the constants are finite., as a consequence of our~Theorems~\ref{thm:HLPP}-\ref{thm:ELPP}. 
Let us deal only with the H\"older case, and let us set $A=1$,$\lambda=1$ for simplicity. The proof is standard but we include it for the sake of completeness.
Thanks to \eqref{eq:scalingL} below, we get that $\cL_{1}(t)$ had the same distribution as $ \cL_{t^{1+\gamma}}(1)$, and to prove that the constant  in \eqref{subadditive} is finite it therefore suffices to show that $\limsup_{\rho\to\infty} \rho^{-1/(1+\gamma)}\cL_{\rho} (1) < +\infty$ a.s. 
First, removing the  point-to-point constraint gives that
$\cL_{\rho} (1)\le \cL_{\gU_\rho}^{(\cH_\gamma^A)}(\Lambda_{1,\infty})$, where the latter is the H$_{\gamma}$-LPP in the domain $\Lambda_{1,\infty} = [0,1]\times \bbR$ with a set $\gU_{\rho}$ which is a Poisson point process of intensity $\rho$, see Section~\ref{sec:HLPP}. 
We cannot directly apply Theorem~\ref{thm:HLPP} because $\Lambda_{1,\infty}$ is not bounded and $\gU_\rho$ does not have a fixed number of points. However, we can write $\cL_{\gU_\rho}^{(\cH_\gamma^A)}(\Lambda_{1,\infty}) = \lim_{j\to \infty} \cL_{\gU_\rho}^{(\cH_\gamma^A)}(\Lambda_{1,j})$ with $\Lambda_{1,j}=[0,1]\times [-j,j]$, 
so that for any $v>0$
\[
\bbP\big( \cL_{\gU_\rho}^{(\cH_\gamma^A)}(\Lambda_{1,\infty}) \ge v \rho^{1/(1+\gamma)} \big) = 
\lim_{j\to\infty} \bbP \big(\cL_{\gU_\rho}^{(\cH_\gamma^A)}(\Lambda_{1,j}) \ge v \rho^{1/(1+\gamma)} \big)\, .
\]
Then, we denote $N_j^{(\rho)}:=|\gU_{\rho} \cap \Lambda_{1,j}|$ the number of Poisson points in $\Lambda_{1,j}$: using Theorem~\ref{thm:HLPP} (with $m=4 \rho j$), we can write
\begin{align}
\label{eq:depoisson}
\bbP \big(\cL_{\gU_\rho}^{(\cH_\gamma^A)}(\Lambda_{1,j}) \ge v \rho^{1/(1+\gamma)} \big)  & \le \bbP\big( N_j^{(\rho)} \ge 4\rho j \big) + \bbP\big( \cL_{4\rho j} (\Lambda_{1,j}) \ge v \rho^{1/(1+\gamma)} \big) \\
& \le  \bbP\big( N_j^{(\rho)} \ge 4\rho j \big) +  \Big(\frac{ 4 c_1}{ v^{1+\gamma}} \Big)^{ v\rho^{1/(1+\gamma)}  } \, . \notag
\end{align}
The first probability goes to $0$ as $j\to \infty$ ($N_j^{(\rho)}$ is a Poisson r.v.\ of parameter $2\rho j$), so that choosing $v_0=(8c_1)^{1/(1+\gamma)}$, we obtain that 
\[\bbP\big( \cL_{\gU_\rho}^{(\cH_\gamma^A)}(\Lambda_{1,\infty}) \ge v_0 \rho^{1/(1+\gamma)} \big) \le 2^{ - v_0 \rho^{1/(1+\gamma)}} \, ,\]
which concludes the argument.

To show the scaling relation \eqref{eq:scalingC}, we consider two different scaling relations satisfied by $\cL_{\lambda}^{(\cH_{\gamma}^{A})}$ and $ \cL_{\lambda}^{(\cE_{a,b}^{B})}$.
For this purpose, we start by considering the following maps: 

(i) $(t,x)\mapsto ( \lambda^{1/(1+\gamma)}t, \lambda^{\gamma/(1+\gamma)}x)$, which does not change the $\gamma$-H\"older  norm of a set $\Delta$;

(ii) $(t,x)\mapsto ( \lambda^{ a/(a+b+1)}t, \lambda^{(b+1)/(a+b+1)} x)$, which multiplies the entropy of a set $\Delta$ (and~$t$) by $\lambda^{a/(a+b+1)}$.

\noindent 
Therefore, since the image of $\gU_{\lambda}$ through these maps has the distribution of $\gU_{1}$, we obtain the following identities in distribution
\begin{equation}
\label{eq:scalingL}
\cL_{\lambda}^{(\cH_{\gamma}^{A})}(t) \stackrel{(d)}{=} \cL_{1}^{(\cH_{\gamma}^{A})} \big( \lambda^{1/(1+\gamma)} t\big) \quad 
\text{ and } \quad \cL_{\lambda}^{(\cE_{a,b}^{B})}(t,tu) \stackrel{(d)}{=} \cL_{1}^{(\cE_{a,b}^{B }
)} \big( \lambda^{a/(a+b+1)} t\big) \, .
\end{equation}
As a consequence, by using \eqref{subadditive}, we also get the existence of the following limits, for any fixed $t>0$, $A,B>0$
\begin{equation}
\label{convergence}
\lim_{\lambda\to\infty}  \frac{1}{\lambda^{1/(1+\gamma)}} \cL_{\lambda}^{(\cH_{\gamma}^{A})} ( t  ) = t \mathtt{C}_{1,A}^{\rm H}  ; \qquad 
  \lim_{\lambda\to\infty}  \frac{1}{\lambda^{a/(a+b+1)}} \cL_{\lambda}^{(\cE_{a,b}^{B})} ( t ) = t  \mathtt{C}_{1,B}^{\rm E} \, .
\end{equation}
Note that we recover the same order for $\cL_\lambda$ as in Theorems~\ref{thm:HLPP}-\ref{thm:ELPP}. 
From \eqref{convergence} we directly obtain that 
\begin{equation}\label{eq:CHE1}
\mathtt{C}_{\lambda, A}^{\rm H}=
\lambda^{1/(1+\gamma) }
\mathtt{C}_{1, A}^{\rm H}\, , \qquad
\text{and}\qquad \mathtt{C}_{\lambda, B}^{\rm E}=\lambda^{a/(a+b+1)}\mathtt{C}_{1, B}^{\rm E} \, .
\end{equation}

Applying another scaling, we can also reduce to the case where $A=1$, $B=1$. We consider the following maps, that preserves the distribution of $\gU_\lambda$:

(i) $(t,x)\mapsto ( A^{1/(1+\gamma)}t, A^{-1/(1+\gamma)}x)$, which divides  the $\gamma$-H\"older  norm  by $A$;

 (ii) $(t,x)\mapsto ( B^{ 1/(a+b+1)}t, B^{-1/(a+b+1)} x)$, which multiplies the entropy by $B^{-1} \times B^{1/(a+b+1)}$. 
 
\noindent
Then, we obtain that
\begin{align*}
\cL_{\lambda}^{(\cH_{\gamma}^{A})}(t) \stackrel{(d)}{=} \cL_{\lambda}^{(\cH_{\gamma}^{1})} \big(  A^{1/(1+\gamma)}t \big)
\qquad
\text{ and } \qquad  \cL_{\lambda}^{(\cE_{a,b}^{B})}(t) \stackrel{(d)}{=} \cL_{\lambda}^{(\cE_{a,b}^{1})} \big( B^{1/(a+b+1)} t\big) \, .
\end{align*}
As a consequence, we have that 
\begin{equation}\label{eq:CHE2}
\mathtt{C}_{1,A}^{\rm H} = A^{1/(1+\gamma)}\mathtt{C}_{1,1}^{\rm H} , \qquad \text{and} \qquad  \mathtt{C}_{1,B}^{\rm E} = B^{1/(a+b+1)} \mathtt{C}_{1,1}^{\rm E}.
\end{equation} 
In the end, \eqref{eq:CHE1}, together with \eqref{eq:CHE2}, give \eqref{eq:scalingC}.
\end{proof}

\begin{remark}\rm
\label{rem:converge}
When considering  $t=1$ with $\lambda=m$, this corresponds to the LPP problem in $\Lambda_{1,\infty}$ with $\gU$ a Poisson point process of intensity $m$. In principle, one could use \eqref{convergence} (with $\lambda=m$), together with a de-Poissonization argument (cf.~\cite{Ha72}), in order to prove the convergence for the point-to-point version of the H$_{\gamma}$-LPP and E-LPP of Sections~\ref{sec:HLPP}-\ref{sec:ELPP} to the constants in \eqref{subadditive}. 
However, the argument should fail (and the constants differ) when the transversal fluctuations of the optimal path become too large: indeed, restricting the paths to stay in a box $[0,1]\times [-1,1]$ is then an important constraint. 
\end{remark}

\section{Non-directed LPP}
\label{sec:nondir}
The notion of compatible-LPP allows for even more general constraints, and for example enables us to deal with non-directed cases.
Let us consider a natural framework, as an example: we work with a time horizon $[0,t]$, and define the Entropy of a subset $\Delta = (x_i)_{1\le i\le k}$ of $\bbR^2$ (the points are considered in a given order), by considering the optimal Entropy of a path going through the points of $\Delta$ (in the correct order) in a time horizon~$t$: 
\begin{align}
\mathbf{Ent}_{a,b}(t,\Delta) &:= \inf \bigg\{ \sum_{i=1}^k  \frac{\|x_{i}-x_{i-1}\|^a}{|t_i-t_{i-1}|^b} \, ; \, t_1<\cdots<t_k\ \text{ subdivision of } [0,t] \bigg \} \, ,
\label{nondir-Ent}
\end{align}
where  $\|\cdot\|$ denotes the Euclidean norm on $\bbR^2$. Another way of presenting it is by saying that  $\mathbf{Ent}_{a,b}(t,\Delta)$ is smaller than $A$ if and only if there exists a path $s:[0,t]\to\bbR^2$ collecting the points of $\Delta$ which has entropy smaller than $A$.
We notice right away that we are able to identify the optimal subdivision $0\le t_1<\cdots <t_k \le t$ used by a path to collect all points of $\Delta$:
the optimal choice is $t_i-t_{i-1} = t \| x_{i}- x_{i-1} \|^{a/(b+1)}  \big(\sum_{i=1}^k \| x_{i}- x_{i-1} \|^{a/(b+1)}  \big)^{-1}$. Then we obtain that the Entropy of $\Delta$ is
\begin{equation}
\label{def:Entnondir}
\mathbf{Ent}_{a,b}(t,\Delta) = \frac{1 }{t^{b}}  \Big( \sum_{i=1}^k \| x_{i}- x_{i-1} \|^{a/(b+1)}  \Big)^{b+1}\, .
\end{equation}

Here again, the case $b>0$ with $a=b+1$ will be of particular interest for us, since it arises naturally from a LDP for non-directed random walks to visit a certain set of points.
In the case $a=b+1$, the optimal choice for the subdivision is to take $t_i-t_{i-1}$ proportional to the distance between the points, and \eqref{def:Entnondir} corresponds to the $(b+1)$-th power of the length of the linear interpolation of the points of $\Delta$---it can then easily be extended to continuous curves  $s:[0,t]\to \bbR^2$.
The case $b=0$ also arises when considering random walks with increments with a stretch-exponential tail, and correspond to the $a$-variation norm of a curve $s:[0,t]\to \bbR$.

%

We will work with the domain $\Lambda_r= \{x\in \bbR^2, \|x\|\le r\}$, for $r>0$ (this choice is not crucial). For $m\ge 1$, we let $\gU_m$ be a set of $m$ independent variables uniform in $\Lambda_r$.
For some fixed $B>0$, the non-directed \emph{Entropy compatible} sets with time horizon $[0,t]$ is defined by
\[\sE_{a,b}^{t,B} = \big \{\Delta \subset \bbR^2  \, ; \, \mathbf{Ent}_{a,b}(t,\Delta) \le B \big \}\, ,\]
and finally the  \emph{non-directed} LPP is
\begin{equation}
\label{de:nondirectedLPP}
\sL_m^{(\sE_{a,b}^{t,B})} (\Lambda_r) := \sup \Big\{  |\Delta| \, ; \, \Delta \subset \gU_m , \Delta \in  \sE_{a,b}^{B}(t) \Big\}. 
\end{equation}
(We use a curly font for  $\sL$ and $\sE$ to visually mark the difference with the directed LPPs.)
We prove the following result, for non-directed LPP.
\begin{theorem}
\label{thm:nondir}
There exist constants $c_5,c_6$ such that for any $t,r$ and $B$, for any $1\le k\le m$  
\begin{align}
\label{nondirE-upper}
\bbP\Big( \sL_m^{(\sE_{a,b}^{t,B})} (\Lambda_r) \ge k\Big)  &\le \Big(  \frac{ c_5 (B t^{b}/r^a)^{2/a} m }{ k^{2(b+1)/a}} \Big)^k \, , \\
\bbP\Big( \sL_m^{(\sE_{a,b}^{t,B})} (\Lambda_r) \le k\Big)  &\le  e^{-c_6 m} + \exp\bigg \{ c_6 k \Big( 1 -  c_{6}\frac{m^{a/2(b+1)}}{k} \big( B t^b /r^a \big)^{1/(b+1)}   \Big)  \bigg\}\, .
\label{nondirE-lower}
\end{align}
Finally, there is some $C>0$ such that
$\bbP$-a.s.\ there is some $m_0$ such that 
\begin{equation}\label{37nondir}   \frac{1}{C} \le  \frac{ \sL_m^{(\sE_{a,b}^{t,B})} (\Lambda_r) }{ m\wedge (B t^{b}/r^a)^{\frac{1}{b+1}}  m^{\frac{a}{2(b+1)}}} \le C \qquad \text{ for all } m\ge m_0\, .\end{equation}
\end{theorem}
We have that $\sL_m$ is of order $m^{\kappa}$ with $\kappa = \frac{a}{2(b+1)} \wedge 1$: we also expect that $\sL_m/m^{\kappa}$ converges a.s.\ to a constant as $m\to\infty$. 
Let us highlight the fact that we find
 $\kappa=1/2$ (as for the standard LPP) in the case $a=b+1$.

\begin{remark}\rm
We stress that we could have defined a corresponding non-directed $\gamma$-H\"older-LPP, the analogous of \eqref{def:Entnondir} ending up being
\[\mathbf{H}_{\gamma}(t,\Delta) = \frac{1 }{t^{\gamma}}\Big( \sum_{i=1}^k \| x_{i}- x_{i-1} \|^{1/\gamma} \Big)^{\gamma} \, .\]
It is very similar to \eqref{def:Entnondir}, with $\gamma = (b+1)/a$ (possibly changing the constants $t,B$): one finds that the corresponding non-directed LPP $\sL_m$ is of order $m^{\kappa}$ with $\kappa = \frac{1}{2\gamma} \wedge 1$ (hence $\kappa=1/2$ for a Lipschitz constraint, as in the original Hammersley LPP).
\end{remark}

\section{Some applications of the  (entropy) constrained LPP}
\label{sec:applications}

Our main goal has been to introduce a generalized LPP, and the results we stated give the first properties of such models, which are already useful in some contexts---it has proven useful in~\cite{cf:BT_HT}.
We now present briefly two applications of the directed and non-directed LPPs, to the context of polymer models.

\subsection{Application I: a model for a directed polymer in Poissonian environment}
\label{sec:appli1}

We define here a very natural variational problem, which encapsulate the energy-entropy competition inherent to models of polymers in random environment.
The random environment is given by a Poisson point process $\gU_{\lambda}$ on $\bbR_+\times \bbR$ of intensity $\lambda>0$ (its law is denoted $\bbP$), and for $\gb>0$, we define the following (point to point) variational problem
\begin{equation}
\label{def:polym1}
\mathcal{Z}_{\lambda,\beta}(t) := \sup_{s:[0,t]\to \bbR, s(0)=s(t)=0 } \Big\{ \beta \, \big| s\cap \gU_{\lambda} \big| - \ent(s)  \Big\} \, ,  
\end{equation}
with $\ent(s)$ defined as in \eqref{def:ent}---because $\gU_{\lambda}$ is countable, $\ent(s)$ is well-defined.
This variational problem constitute a simplified model to study the energy-entropy competition of polymer models,  $| s\cap \gU_{\lambda}|$ being viewed as a measure of the energy of a trajectory $s$
Again, the central cases that we have in mind is when $a=b+1$ or $b=0$ in the definition of the entropy \eqref{def:ent}, see Remark~\ref{rem:pathentropy} (when the entropy derives from the LDP of a simple random walk, we have $a=2, b=1$).
The idea of this model is similar to that of \cite{CY13} which considers a Brownian polymer in Poissonian medium. However, here, we only consider trajectories maximizing the energy-entropy balance, and we also allow for more general entropy than that of the Brownian motion (for which $a=2,b=1$).

{ 
Let us stress that the variational problem \eqref{def:polym1} has already appeared in \cite{BCK14} (in the case $a=2,b=1$) as a solution for a Hamilton-Jacobi equation used to study the stationary solutions of a Burgers equation (with a forcing induced by the points of a PPP). It has proven to be useful for the study of the thermodynamic limit for directed polymers, cf.~\cite{BL16}. 
}

First of all, we notice that $\mathcal{Z}_{\lambda,\gb}(t)$ is a super-additive ergodic sequence, so that Kingman's sub-additive  ergodic theorem gives that the limit
\begin{equation}
f(\lambda,\gb) := \lim_{t\to\infty} \frac{1}{t} \mathcal{Z}_{\lambda,\gb}(t)
\end{equation}
exists a.s.\ and in $L^1$, and is  $\bbP$-a.s.\ constant.
The fact that  $f(\lambda,\gb)$ is finite derives from our estimates in Theorem~\ref{thm:ELPP} (together with the argument in Section~\ref{sec:poisson}, see \eqref{eq:depoisson}), so we skip it---we mention that this fact was an important part of the study in \cite{BCK14}. 

We also have scaling relations for $\mathcal{Z}_{\lambda,\gb}(t)$. Indeed, consider the two following maps: (i) $(t,x) \mapsto (\lambda^{-a/(a+b)} t, \lambda^{-b/(a+b)} x)$ whose image of $\gU_{\lambda}$ has distribution $\gU_{1}$ and which preserves the entropy; (ii) $(t,x) \mapsto ( \gb^{-1/(a+b)} t, \gb^{1/(a+b)}x)$, which multiplies the entropy by $\gb$, while preserving the distribution of $\gU_{\lambda}$. We therefore obtain that
\begin{equation}
\label{scalingZ}
\mathcal{Z}_{\lambda,\gb}(t) \stackrel{(d)}{=} \mathcal{Z}_{1,\gb} \big( \lambda^{-a/(a+b)}  t \big) \qquad \text{ and } \qquad \mathcal{Z}_{\lambda,\gb}(t) \stackrel{(d)}{=}  \gb \mathcal{Z}_{\lambda,1} \big( \gb^{-1/(a+b)} t \big) \, .
\end{equation}
A first consequence is that we get that $f(\lambda,\gb) = (\gb^{a+b+1}\lambda^{a})^{1/(a+b)} f(1,1)$, where $f(1,1)$ is a constant that needs to be determined.
Another consequence is that, if we consider the alternative problem where  we take $\lambda \to\infty$ (instead of $t\to\infty$), we get that, for any fixed positive $t,\gb$, the limit
\begin{equation}
\label{convZ}
 \lim_{\lambda\to +\infty}\frac{1}{\lambda^{a/(a+b)}} \mathcal{Z}_{\lambda,\gb}(t)  = t f(1,\gb) = t \gb^{(a+b+1)/(a+b)} f(1,1) 
\end{equation}
exists a.s.\ and in $L^1$.

For this model, some important questions  remain unanswered:
(i) what is the value of the constant $f(1,1)$?
(ii) what does the maximizer of $\mathcal{Z}_{\lambda,\gb}(t)$ look like? For instance, what is its typical transversal fluctuation exponent? We mention that in \cite{BCK14,BL16}, the results are mostly qualitative, such as the existence and coalescence of semi-infinite maximizers for this model.

\subsection{Application II: non-directed polymers in heavy-tail environment}
\label{sec:appli}

The directed E-LPP  have already proved to be useful to understand the transversal fluctuations and scaling limits of directed polymers in heavy-tail random environment, see~\cite{cf:BT_HT}. The continuous limit of the model is found to be an energy-entropy variational problem, and E-LPP appears central to ascertain its well-posedness.
Here, we define an analogous variational problem in the non-directed setting, and show that it is well defined. It should also appear as the scaling limit of some non-directed polymer model in heavy-tail random environment.

As a continuum disorder field, we let $\cP:=\{(w_i,x_i,y_i)\colon i\geq 1\}$ be a PPP on  $[0,\infty)\times \mathbb R^2 $,
of intensity $\mu(\dd w \dd x \dd y)=\frac{\alpha}{2} w^{-\alpha-1}\ind_{\{w>0\}}\dd w\dd x \dd y$ ---it derives from the scaling of a discrete field of disorder with heavy-tail distribution, with tail exponent $\alpha$.
For a continuous path $s:[0,1]\to \bbR^2$, we can then define its  energy by summing the weights in $\cP$ ``collected'' by $s$,
$\pi(s) = \sum_{(x_i,y_i)\in s} w_i$. We can also define its length $\len(s) = \int_0^1 \| s'(u)\| \dd u$, and  we consider $\ell(s)^{\nu}$ for some $\nu > 1$ as a measure of its entropy. Indeed,  if $s$ is a linear interpolation of a 
finite number of points in $\cP$, then $\ell(s)^\nu$ is nothing but the non-directed entropy defined in \eqref{def:Entnondir} with $a=b+1$ and $b+1=\nu$.
This choice derives from LDP for a random walk, and $\nu=2$ corresponds to the moderate deviation regime of the simple random walk.

Thanks to the non-directed LPP of Section \ref{sec:nondir}, we are able to show that the energy/entropy variational problem is well defined, when $\ga \in (2/\nu,2)$.
\begin{proposition}
\label{prop:polym}
For any $\nu>1$, the following variational problem is well defined for all $\gb\ge 0$, when $\ga\in (2/\nu,2)$,
\begin{equation}
\label{eq:Tgb}
\cT_{\beta}^{(\nu)} := \suptwo{ s:[0,1]\to \bbR^2 }{s(0)=0, \, \len(s) <\infty}  \big\{  \beta \pi(s) - \len(s)^{\nu} \big\} \, . 
\end{equation}
For $\gb>0$, we have that $\cT_{\beta}^{(\nu)}>0$ a.s.\ and  $\bbE[ ( \cT_{\gb}^{(\nu)})^{\kappa}] <\infty$ for any $\kappa <\ga-2/\nu$. 
Moreover, for any $\ga\in ( 2/\nu  ,2)$, we have the scaling relation
$
 \cT_{\gb}^{(\nu)} \stackrel{(\dd)}{=} \gb^{\tfrac{\nu\ga}{\nu\ga -2}}\,  \cT_{1}^{(\nu)} .
$
On the other hand, if $\ga\in (0,2/\nu]$, we have that $\cT_\beta^{(\nu)}=+\infty$ a.s.
\end{proposition}
The proof follows exactly the same scheme as that in \cite[Section 4]{cf:BT_ELPP}, with Proposition~\ref{prop:polym} in place of \cite[Theorem~2.4]{cf:BT_ELPP} so we skip it.

Polymers in random environment have mostly been considered in the directed framework, see \cite{Comets} for a thorough review, or in the semi-directed context of stretched polymers, see \cite{IV12,Z13}, or \cite{IV12_review} for a review.
Proposition~\ref{prop:polym} therefore shows that our generalized LPP can be useful to study non-directed polymers: the variational problem can be thought as an energy/entropy model for a polymer in continuous random environment.
The most natural question is now to consider a (discrete) non-directed polymer model in random environment (the Hamiltonian being the sum of the weights of the sites visited by the random walk), and prove its convergence  to \eqref{eq:Tgb}, in the case of a heavy-tail environment {(as done in \cite{cf:BT_HT})}.

\section{Proofs of the constrained LPP results} 
 
We prove here Theorems~\ref{thm:HLPP}-\ref{thm:ELPP}-\ref{thm:nondir}. The almost sure statements are straightforward applications of the first parts of the theorems (via the Borel-Cantelli lemma), so we skip their proof. 
We write the details only for the E-LPP, the H-LPP results following exactly the same scheme---for the upper bound, the ideas are similar to those developed in \cite[Part~1]{cf:BT_ELPP}, in a special case of the E-LPP. Some more technical details are needed to obtain~\eqref{nondirE-lower}.

\subsection{Entropy-constrained LPP}
\label{sec:Ent}

We prove first \eqref{ELPP-upper}, and then \eqref{ELPP-lower}.

\subsubsection{Upper bound}
Define $E_{k}(t,B)$ the set of $k$ (ordered) elements up to time-horizon $t$ that  have an entropy bounded by $B$:
\[E_{k}(t,B) = \Big\{ (t_i,x_i)_{1\le i\le k} \, ;\,  0<t_1<\cdots<t_k< t  , \ent_{a,b}\big( (t_i,x_i)_{1\le i\le k} \big) \le B\Big\} \, .\]
We are able to compute exactly the volume of $E_k (t,B)$.
\begin{lemma}
\label{lem:VolE}
For any $t>0$ and $B> 0$, we have for any $k\ge 1$
\[\mathrm{Vol} \big( E_k (t,B) \big) =  2^k  \big( \tfrac1a \big)^k    \frac{\Gamma( \tfrac1a )^k }{\Gamma \big( \tfrac{k}{a}+1 \big)} \frac{\Gamma( \tfrac{a+b}{a})^k}{\Gamma\big(k\tfrac{(a+b)}{a} +1 \big)} \times B^{k/a} t^{k (a+b)/a} .\]
In particular, it gives that there exists some constant $C=C_{a,b}$ (during the course of the proof, one finds  that $C_{a,b} \le c (a+b)^{-1/2}$)  such that 
\[\mathrm{Vol}(E_k(t,B))  \leq \Big( \frac{ C B^{1/a} t^{(a+b)/a} }{k^{(a+b+1)/a}}  \Big)^k\, .\]
\end{lemma}
\begin{proof}
Again, using a decomposition over the left-most point in $E_k(t,B)$ at position $(u,y)$ (by symmetry we can assume $y\geq 0$): it leaves $k-1$ points with remaining time horizon $t-u$ and constraint $B- \tfrac{|y|^a}{ u^{b}}$, we obtain the key induction formula below
\begin{equation}
\label{eq:induction}
\mathrm{Vol}\big( E_k(t,B) \big) = 2 \int_{u=0}^t \int_{y=0}^{(B u^b)^{1/a}} \mathrm{Vol }\Big( E_{k-1}(t-u , B - \tfrac{y^a}{ u^{b}} ) \Big) dy du.
\end{equation}
We give the details of the induction for the sake of completeness, but the proof is a straightforward calculation.
First of all, we have for $k=1$
\[\mathrm{Vol}(E_1 (t,B) )  = 2 \int_{u=0}^t \int_{y=0}^{(B u^{b})^{1/a}}  du dy = 2 B^{1/a} \int_0^t u^{b/a} du = 2 B^{1/a} \frac{a}{a+b} t^{(a+b)/a} \, .\]

For $k\geq 2$, by induction, we have
\begin{align*}
\mathrm{Vol}\big( E_k(t,B) \big) = 2^{k-1} \big( \tfrac{1}{a} \big)^{k-1}  & \frac{ \Gamma( \tfrac1a )^{k-1}  }{\Gamma\big( (k-1)/a+1 \big )}  \frac{\Gamma( \tfrac{a+b}{a})^{k-1}}{\Gamma\big((k-1)\tfrac{(a+b)}{a} +1 \big)} \\
&\times \int_{u=0}^t \int_{y=0}^{(B u^b)^{1/a}}   (t-u)^{(k-1)(a+b)/a} \big( B-\tfrac{y^a}{u^b} \big)^{(k-1)/a}  dy du.
\end{align*}

Then, by a change of variable $z= y^a /(Bu^b)$, we get that
\begin{align*}
\int_{y=0}^{(B u^b)^{1/a}} \big( B-\tfrac{y^a}{u^b} \big)^{(k-1)/a}  dy
&= B^{(k-1)/a} \int_0^1 (1-z)^{(k-1)/a} \, \tfrac1a z^{1/a-1} B^{1/a} u^{b/a} dz  \\
& = \tfrac1a  \, A^{k/a} u^{b/a}\,  \frac{\Gamma\big( (k-1)/a+1 \big) \Gamma(1/a)}{\Gamma(k/a)}\, .
\end{align*}

Moreover, we also have, with a change of variable $w=u/t$
\begin{align*}
 \int_{u=0}^t  u^{b/a} (t-u)^{(k-1)(a+b)/a} du &= t^{(k-1)(a+b)/a+ b/a + 1} \int_0^1 w^{b/a} (1-w)^{(k-1)(a+b)/a} dw \\
 &=  t^{k(a+b)/a} \frac{\Gamma\big(b/a+1\big) \Gamma \big( (k-1)(a+b)/a +1 \big)}{ \Gamma\big( k(a+b)/a +1 \big) }\, ,
\end{align*}
and this completes the induction.

For the inequality in the second part of the lemma, we use Stirling's formula $\Gamma(1+x) \sim \sqrt{2\pi x} (x/e)^{x}$ as $x\to\infty$ to control $\Gamma\big( k (a+b)/a +1\big)$ and $\Gamma\big( k/a +1\big)$, and we obtain
\[
\mathrm{Vol}(E_{k} (t,B)) \leq \frac{c}{k \sqrt{1/a} \sqrt{(a+b)/a}} \bigg(  \frac{ \frac{2}{a}   \Gamma( \tfrac1a ) \Gamma( \tfrac{a+b}{a}) \times B^{1/a} t^{(a+b)/a}}{ \big(e^{-1}/a \big)^{1/a} \big( e^{-1} (a+b)/a \big)^{(a+b)/a}  k^{1/a} k^{(a+b)/a}} \bigg)^k \, .  
\]
Thanks to the asymptotics of $\Gamma(\ga)$ as $\ga\to +\infty$ and $\ga \to 0$, we find that there is a constant $c$ such that for all $a,b$
\[
\mathrm{Vol}(E_{k} (t,B)) \leq \frac{a}{  \sqrt{a+b} } \bigg(  \frac{ c B^{1/a} t^{(a+b)/a}}{ (a+b)^{1/2}  k^{1/a} k^{(a+b)/a}} \bigg)^k \, .  
\]
\end{proof}

We then use Lemma~\ref{lem:VolE} to control the probability that $\cL_m^{(\cE_{a,b}^{B})} (\Lambda_{t,x})$ is larger than~$k$:
\begin{equation}
\label{eq:firstmoment}
\bbP \Big( \cL_m^{(\cE_{a,b}^{B})} (\Lambda_{t,x}) \geq k \Big) = \bbP(\cN_k \geq 1) \leq \bbE[\cN_k] ,
\end{equation}
where $\cN_k = {\rm Card} \{ \Delta \subset \gU_m \, ; \, \Delta \in \cE_{a,b}^{B}\}$ is the number of sets of $k$ points in $\gU_m$ that are $\cE_{a,b}^{B}$-compatible.
Since all the points of $\gU_m = \{Z_1, \ldots, Z_m\}$ are exchangeable, we have
\[\bE[\cN_k] = \binom{m}{k} \bbP\Big(  \exists \ \sigma\in \mathfrak{S}_k\ s.t. \  (Z_{\sigma(1)},\ldots,  Z_{\sigma(k)}) \in E_{k}(t,B) \Big)\, .\]
Since the $(Z_i)_{1\le i\le m}$ are i.i.d.\ uniform in $\Lambda_{t,x} = [0,t]\times [-x,x]$ (of volume $2tx$), we get that
\begin{equation}
\label{eq:exchangeable}
\bbE[N_k] = \binom{m}{k} \times \frac{{\rm Vol}\big( E_k(t,B) \big)}{ (2tx)^k /k!} ,
\end{equation}
where the  $k!$ comes from the fact that we rearrange the $Z_i$'s so that $0<t_1< \cdots < t_k <t$.
Using Lemma \ref{lem:VolE} together with $\binom{m}{k} \le m^k/k_1$, we therefore obtain that
\begin{equation}
\label{conclusion}
\bbP \Big( \cL_m^{(\cE_{a,b}^{B})} (\Lambda_{t,x}) \geq k \Big) \le  \binom{m}{k} \times \frac{{\rm Vol}\big( E_k(t,B) \big)}{ (2tx)^k /k!} \le  \bigg( \frac{CB^{1/a} t^{(a+b)/a} m }{ tx k^{(a+b+1)/a}}\bigg)^k  \, .
\end{equation}
This gives the upper bound \eqref{ELPP-upper}.

\subsubsection{Lower bound}

For any $k\ge 1$, consider  for $1\le i\le 4k$ the sub-boxes of $\Lambda_{t,x}$
\[\cB_i := \Big[\frac{(i-1)t}{4k},\frac{it}{4k}\Big) \times \Big[ - \frac{B^{1/a}  (t/4)^{b/a}}{2k^{(b+1)/a}} \wedge x , \frac{B^{1/a}  (t/4)^{b/a}}{2k^{(b+1)/a}} \wedge x\Big] \, .\]
Then, notice that if there are at least $k$ boxes among $\{\cB_{2i} \}_{1\le i\le 2k}$ containing (at least) one point, then this set of $k$ points has an entropy which is bounded by 
\[k\times \frac{(B^{1/a}  (t/4)^{b/a} k^{-(b+1)/a})^a}{(t/4k)^b} \le B.\]
Hence, we get that
\begin{equation}
\label{eq:firststeplower}
\bbP\Big( \cL_m^{(\cE_{a,b}^{B})}(\Lambda_{t,x}) \ge k \Big) \le \bbP\Big( \sum_{i=1}^{2k} \ind_{ \{|  \gU_m \cap \cB_{2i}  | \ge 1 \}} \le k \Big) = \bbP\Big(\sum_{i=1}^{2k} \ind_{ \{|  \gU_m \cap \cB_{2i}  | =0 \}} \le k \Big)\,.
\end{equation}
For the last probability, we use a union bound and the fact that the  $\ind_{ \{|  \gU_m \cap \cB_{2i}  | =0 \}} $ are exchangeable, to get that
\begin{align}
\bbP\Big( \cL_m^{(\cE_{a,b}^{B})}(\Lambda_{t,x}) \le k \Big) &\le  \binom{2k}{k} \bbP\Big(  \gU_m  \cap \bigcup_{i=1}^k \cB_i  = \emptyset \Big) \notag\\
&\le2^{2k} \Big( 1- \frac{ B^{1/a} t^{b/a} }{4^{b/a} k^{(a+b)/a} x} \wedge\frac14 \Big)^m \, .
\end{align}
In the second inequality we used that  $ \gU_m$ is a set of $m$ independent random variables uniform in $ \Lambda_{t,x}$ (of volume $2tx$), and that $\bigcup_{i=1}^k \cB_i$ has  volume $ \tfrac{B^{1/a} (t/4)^{(a+b)/a} }{k^{(b+1)/a}} \wedge \tfrac{tx}{2}$. 
Therefore, using also that $1-x \le e^{-x}$ we obtain that
\begin{equation}
\label{conclusionlower}
\bbP\Big( \cL_m^{(\cE_{a,b}^{B})}(\Lambda_{t,x})\le k \Big) \le \exp\bigg\{ c k \Big( 1- c \Big( \frac{ B^{1/a} t^{b/a} }{ x k^{(a+b)/a} }\wedge 1   \Big)  \frac{m}{k} \Big) \bigg\} \, ,
\end{equation}
which concludes the proof of the \eqref{ELPP-lower}.

\subsection{H\"older-constrained LPP}
\label{sec:Hol}

We only give a brief outline of the proof.

\smallskip
\noindent
{\it Upper bound.}
We define 
\[H_{k}(t,A) = \Big\{ (t_i,x_i)_{1\le i\le k} \, ;\,  0<t_1<\cdots<t_k< t  , \mathrm{H}_{\gamma}\big( (t_i,x_i)_{1\le i\le k} \big) \le A\Big\} \, .\]
Then, as above, we are able to compute exactly the volume of $H_k (t,A)$: for any $t>0$ and $A> 0$, we have for any $k\ge 1$
\[\mathrm{Vol} \big( H_k (t,A) \big) =  (2A)^k \frac{\Gamma(1+\gamma)^k}{\Gamma\big( k(1+\gamma)+1\big)}  t^{k (1+\gamma)} .\]
We do not develop the proof, which comes from an induction formula analogous to \eqref{eq:induction}:
\[\mathrm{Vol}\big( H_k(t,A) \big) = 2 \int_{u=0}^t \int_{y=0}^{A u^{\gamma}} \mathrm{Vol }\big( H_{k-1}(t-u , A) \big) dy du .\]
As a consequence of Stirling's formula, there exist constants $c,C$ such that 
\begin{equation}
\label{formulaeVol}
\mathrm{Vol}(H_{k} (t,A)) \leq c\bigg(  \frac{2A \Gamma( 1+\gamma)  t^{1+\gamma}}{  \big( (1+\gamma)/e \big)^{1+\gamma} k^{1+\gamma} } \bigg)^k 
\le  \bigg(  \frac{C A  t^{1+\gamma}}{ (1+\gamma)^{1/2} k^{1+\gamma} } \bigg)^k \, .
\end{equation}
To obtain \eqref{HLPP-upper}, one then follows exactly the same scheme as in \eqref{eq:firstmoment}--\eqref{conclusion} above, using \eqref{formulaeVol} in place of Lemma~\ref{lem:VolE}.

\smallskip
\noindent
{\it Lower bound.} 
For any $k\ge 1$, we consider for $1\le i\le 4k$ the sub-boxes of $\Lambda_{t,x}$: 
\[\cB_i := \Big[\frac{(i-1)t}{4k},\frac{it}{4k}\Big) \times \Big[ - \frac{A(t/k)^{\gamma}}{2} \wedge x , \frac{A(t/k)^{\gamma}}{2} \wedge x\Big] \, .\]
Note that, if there are at least $k$ boxes among $\{\cB_{2i} \}_{1\le i\le 2k}$ containing (at least) one point, then this set of $k$ points has a $\gamma$-H\"older norm which is bounded by $  A (k/t)^{\gamma} / (t/k)^\gamma  \le A$.
Then, we obtain~\eqref{HLPP-lower} exactly as above, see~\eqref{eq:firststeplower}--\eqref{conclusionlower} above.

\subsection{Non-directed E-LPP} 

We proceed analogously to the two previous sections, some details differing for the lower bound.

\smallskip
\noindent
{\it Upper bound.}
We define
\begin{align*}
{\rm E}_k (t,B) & = \Big\{  \Delta = (x_i)_{1\le i\le k} \, ; \,  \mathbf{Ent}_{a,b}(t,\Delta) \le B \Big\} \\
& = \Big\{  \Delta = (x_i)_{1\le i\le k} \, ; \, \sum_{i=1}^k \| x_{i}- x_{i-1} \|^{a/(b+1)}    \le D \Big\} =: \tilde {\rm E}_k\big(  D\big) \, ,
\end{align*}
with $D=(Bt^b)^{1/(b+1)}$ ---we used \eqref{def:Entnondir} to get the second equality. As above, we are able to obtain its volume:
setting $ \gamma= (b+1)/a$ for simplicity, we have for any $D>0$ and any $k\geq 1$
\[\mathrm{Vol} \big(  \tilde {\rm E}_k(D) \big) =  (2\pi \gamma)^{k} \frac{\Gamma(2\gamma)^k}{\Gamma(2k\gamma +1)} D^{2k\gamma} .\]
This is easily proven by iteration,  using the recursion formula, for $k\ge 2$
\begin{align*}
\mathrm{Vol} \big(  \tilde {\rm E}_{k}(D) \big) &= \int_{0}^{D^{\gamma}} 2 \pi r \mathrm{Vol} \big(  \tilde {\rm E}_{k-1}(D - r^{1/\gamma}) \big) dr \, .
\end{align*}
We leave the details to the reader.
Then, an easy application of Stirling's formula gives that there exists some constant $C$ such that 
\begin{equation}
\label{formulaVolnondir}
\mathrm{Vol}({\rm E}_k (t,B) )  \leq \Big( \frac{ C (B t^b)^{2/a} }{k^{2(b+1)/a}}  \Big)^k\, .
\end{equation}
The upper bound  \eqref{nondirE-upper} then comes from the exact same scheme as for \eqref{eq:firstmoment}--\eqref{conclusion} above, using  \eqref{formulaVolnondir} in place of Lemma~\ref{lem:VolE}.

\smallskip
\noindent
{\it Lower bound.} 
The idea of the proof is similar to the one in the directed context, with more technicalities due to the non-directedness. 
We consider a partition of the plane into small squares 
of side $\delta:= \pi r/\sqrt{m}$: for any $x\in (\delta \bbZ)^2$ we let $\cB_x $ be the square
of side $\delta$ centered at $x$. 
It is easy to see that there are at least $m/4$ disjoint squares $\cB_x $ (provided that $m$ is large enough) that can be 
placed into a rectangle (inscribed in $\Lambda_r$) ordered as follows: we let $x_0=0$ and then we enumerate  $x_1,\ldots, x_{m/4}$ following a spiral in a clockwise way, in order to have that any two consecutive $\cB_{x_i},\cB_{x_{i+1}}$ are adjacent (see Figure \ref{fig1}). 

Note that a square $\cB_x$ has volume $\pi^2 r^2/m$ (and recalling $\Lambda_r$ has volume $\pi r^2$), $\cB_x$ contains at least one point of $\gU_m$ with probability $1-(1-\pi/ m)^m \ge 1-e^{-\pi}$.
We define $Q_{m/4}$ the number of non-empty squares among $\cB_{x_0}, \ldots, \cB_{x_{m/4}}$, and we define iteratively the indices~$I_j$ of the non-empty squares, by $I_0=0$ and for $1\le j\le Q_{m/4}$
\[I_{j} = \inf \big\{ i >I_{j-1} \ ;\ \cB_{x_i} \cap \gU_m \neq \emptyset  \}\, .\]

\begin{figure}[htbp]
  \captionsetup{width=0.9\linewidth}
\includegraphics[scale=0.6]{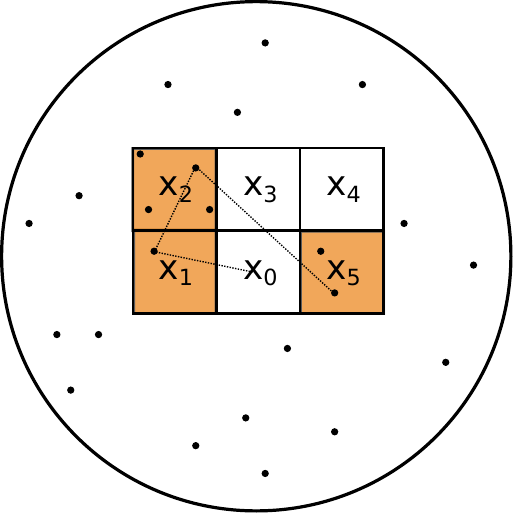}
\caption{
\footnotesize In the picture we put $m=24$ points uniformly on $\Lambda_r$ and we consider a rectangle built by $6 = m/4 $ squares $\cB_{x_0}, \cdots,\cB_{x_5}$ 
enumerated following a spiral in a clockwise way starting from the origin. Then we consider the non-empty rectangles (in orange)
and their indices. In this example we have $I_1=1, I_2=2, I_3=5$. 
Finally we draw a path starting from the origin and collecting one point in exactly all $\cB_{I_1}, \cdots,\cB_{I_3}$.}
\label{fig1}
\end{figure}

For $k\ge1$, and if $Q_{m/4}\ge k$, we may consider a path $\Delta$ collecting one point in exactly all $\cB_{x_{I_1}}, \ldots, \cB_{x_{I_k}}$: the entropy of such $\Delta$ is bounded by (see Figure \ref{fig1})
\[\frac{1}{t^b}  \Big( \sum_{j=1}^k \big(  4(I_{j}-I_{j-1})\delta  \big)^{a/(b+1)} \Big)^{b+1} \le \frac{4^a r^a}{t^b m^{a/2}} \Big(\sum_{j=1}^k U_j \Big)^{b+1},\]
where we  set $U_j:= (I_{j}-I_{j-1})^{a/(b+1)}$.
Therefore,  for $\sL_m^{(\sE_{a,b}^{t,B})} ( \Lambda_r) $ to be smaller or equal than $k$, one needs to have either $Q_{m/4} <k$ or that the entropy of $\Delta$ chosen above is larger than $B$: this leads to
\begin{equation}
\label{lowerbound1}
\bbP \Big( \sL_m^{(\sE_{a,b}^{t,B})} ( \Lambda_r) \leq k \Big) \le \bbP\big( Q_{m/4} <k \big) + \bbP\Big(  Q_{m/4} \ge k\, , \, \sum_{j=1}^k U_j  > \Big( \frac{Bt^b m^{a/2}}{4^a r^a}  \Big)^{1/(b+1)}  \Big) \, .
\end{equation}

For the first term, and for $k\le \gep^2 m/4$ (with $\gep>0$ small, fixed in a moment), we realize that $Q_{m/4} <k$ implies that there are at least $(1-\gep^2) m/4$ empty squares, which gives by a union bound that
\[\bbP\big( Q_{m/4} < k\big)  \le  \binom{m/4}{(1-\gep^2) m/4}  \bbP\Big( \gU_m \cap \bigcup_{i=1}^{(1-\gep^2) m/4} \cB_{x_i} = \emptyset  \Big) \le  e^{ c_{\gep} m} \Big(1-  \frac{(1-\gep^2) \pi}{4}\Big)^m \, . \]
For the second inequality, we used that the volume of $\bigcup_{i=1}^{(1-\gep^2) m/4} \cB_{x_i}$ is $(1-\gep^2)\pi^2 r^2 /4$. We note that the constant $c_{\gep}$ goes to $0$ as $\gep$ goes to $0$: we can therefore fix $\gep>0$ sufficiently small so that
\begin{equation}
\bbP\big( Q_{m/4} < k\big)  \le  e^{- \pi m/8}\quad \text{for all }\  k\le \gep^2 m/4.
\end{equation}

For the second term in \eqref{lowerbound1}, let us write $V:= k^{-1}\big( Bt^b m^{a/2}/r^a  \big)^{1/(b+1)}$ ---we will consider only the case when $V$ is large---, so that we need to bound
\begin{equation}
\label{secondpart}
\bbP\Big( Q_{m/4} \le k \,  ,\, \sum_{j=1}^k U_j  >  k V  \Big) \le \bbP\big( N_k >  \gep m \big) + \bbP\Big( Q_{m/4} \le k \,  , N_k\le  \gep m  \, ,\, \sum_{j=1}^k U_j  >  k V  \Big) \, ,
\end{equation}
where $N_k$ denotes the total number of points in the  non-empty squares $\cB_{x_{I_1}},\ldots, \cB_{x_{I_k}}$.
We easily have that 
\[\bbP\big( N_k >  \gep m  \big) \le \frac{1}{e^{-c k}} \binom{m}{  \gep m } \Big( \frac{\pi k}{m}\Big)^{  \gep m } \le e^{c k} \frac{(\pi k)^{\gep m}}{ (\gep m) !} \, , \]
where the denominator in the first inequality comes from the fact that we work conditionally on the fact that $k$ squares are non-empty (which has probability bounded  below by $e^{- ck}$).
Hence, since we work with  $k\le \gep^2 m/4$, and provided that $\gep$ has been fixed small enough, we get that there is a constant $c>0$ such that $\bbP ( N_k >  \gep m  ) \le e^{-c m}$.

For the last part, note that since the squares $\cB_x$ are exchangeable, we can control for $1\le i_1<\cdots<i_k \le m/4$
\begin{align*}
\bbP\big( I_1=i_1 ,\ldots, I_k =i_k & ; N_k \le \gep m\big) \\
&=  \sumtwo{n_1,\ldots , n_k }{1\le n_1+\cdots+n_k \le \gep m} 
 \binom{m}{n_1,\ldots,n_k} \Big(\frac{\pi}{m}\Big)^{n_1 +\cdots + n_k}
\Big( 1-\frac{\pi i_k }{m}  \Big)^{m - (n_1 +\cdots + n_k)}  \\
&\le \Big( 1-\frac{\pi i_k }{m}  \Big)^{(1-\gep) m} \sumtwo{n_1,\ldots , n_k}{1\le n_1+\cdots+n_k \le \gep m}  \frac{\pi^{n_1}}{n_1!} \cdots   \frac{\pi^{n_k}}{n_k!} \le e^{-(1-\gep) \pi i_k} e^{\pi k} \, .
\end{align*}
Where we used that in order to have $ I_1=i_1,\ldots, I_k =i_k $ there must be exactly  $k$ non-empty squares among the first $i_k$ (with $n_1,\ldots,n_k$ points in them) and $i_k-k$ empty ones. The remaining $m-(n_1+\dots+n_k)$ points must be outside the first $i_k$ squares. 
 For the second inequality, we used that $n_1+\cdots + n_k \le \gep m$, and that the multinomial coefficient is bounded by $m^{n_1+\cdots +n_k} /(n_1! \cdots n_k !)$.
Hence, there is a constant $c$ such that
\[
\bbP\big( I_1=i_1 ,\ldots, I_k =i_k  ; N_k \le \gep m\big) \le e^{ c k} \times \bbP \big( G_j = i_j -i_{j-1} \ \text{ for all } 1\le j\le k\big)\, ,
\]
where $(G_j)_{j\ge 1}$ are i.i.d.\ geometric random variables, of parameter $1-e^{-(1-\gep)\pi}$.
We therefore obtain that, provided that $V$ is large enough
\begin{align*}
\bbP\Big( Q_{m/4} \le k \,  , N_k\le  \gep m  \, ,\, \sum_{j=1}^k U_j  >  k V  \Big)
& \le e^{ck} \bbP\Big( \sum_{j=1}^k (G_j)^{a/(b+1)} >k V \Big)  \le e^{- c' k V}\, .
\end{align*}

\smallskip
To conclude, we have obtained that there are constants such that for $k\le \gep^2 m/4$, and for $V:=k^{-1}\big( Bt^b m^{a/2}/r^a  \big)^{1/(b+1)}$ large enough,
\begin{equation}
\bbP \Big( \sL_m^{(\sE_{a,b}^{t,B})} (\Lambda_r) \leq k \Big) \le e^{-c m } + e^{- c' kV }\, .
\end{equation}
One obtains \eqref{nondirE-lower} by observing that when $V$ is small $e^{-c k(V-1)}$ is larger than $1$. The statements holds for all $k\le m$ by adjusting the constants.

\begin{appendix}

\section{Further simulations and conjectures}
\label{app}

In this appendix, we present some simulations, that help us make some predictions on the values of the constants in \eqref{subadditive}, and support the belief that the model is in the KPZ universality class. We treat only the directed case because in the non-directed case simulations are much more greedy and do not bring any convincing insight---our algorithm could probably be improved, but our goal is simply to hint for some conjectures, and our simulations fill that role perfectly.

\subsection{Directed H$_{\gamma}$-LPP}
\label{app1}

{
For the H$_{\gamma}$-LPP, we performed simulations in the Poissonian context of Section~\ref{sec:poisson}: we work with intensity $\lambda=1$ and constraint $A=1$, so we write $\cL(t)$ for $\cL_{\lambda}^{(\cH_\gamma^{A})}(t)$ to simplify notations.

\begin{figure}[htbp]
  \captionsetup{width=0.9\linewidth}
\includegraphics[width=0.35\linewidth]{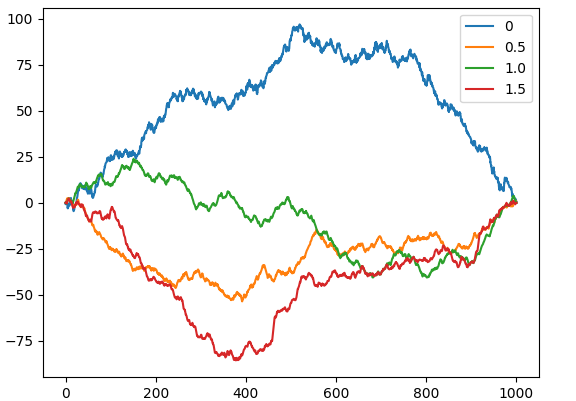}
\vspace{-0.3cm}
\caption{
\footnotesize Optimal paths for the H$_{\gamma}$-LPP with $t=1000$, for different values of $\gamma$. The same set of points is used in all four simulations. For $\gamma=0$ we have here $\cL(t)=2707$, for $\gamma=0.5$ $\cL(t)=1715$, for $\gamma=1$ $\cL(t)=1408$, and for $\gamma=1.5$ $\cL(t)=1238$. 
\label{fig:t1000}}
\end{figure}

\subsubsection{{\rm (1)} Value of the constant}
Let us present here the results of our simulations to test the value of the constant $\mathtt{C}=\mathtt{C}_{1,1}=\lim_{t\to\infty} t^{-1} \cL(t) $ in \eqref{convergence}.  We ran simulations for $t=1000$, in the box $[0,t]\times [-t^{2/3},t^{2/3}]$ (with $\lambda=1$ and $A=1$).

\begin{figure}[htbp]
  \captionsetup{width=0.9\linewidth}
\includegraphics[width=0.35\linewidth]{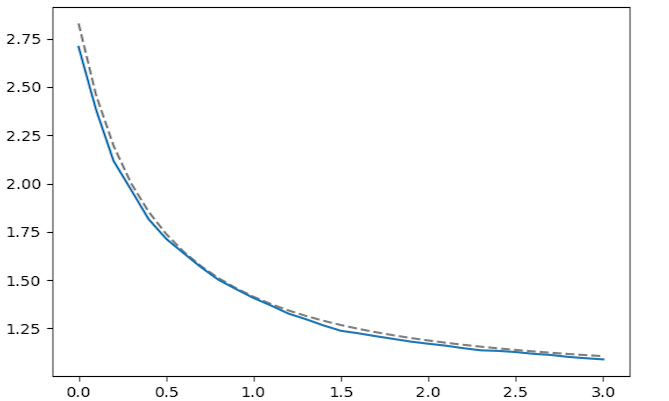}
\caption{\footnotesize Approximated values of the H$_{\gamma}$-LPP constant: the function represents the value of $t^{-1} \cL(t)$ with $t=1000$, for different values of $\gamma \in [0,3]$. 
The dotted grey line 
represents the function $\gamma \mapsto  \frac{2^{3/2}}{1+\gamma}\Gamma(1+\gamma)^{1/(1+\gamma)}$, which seems  to fit quite well the values of $t^{-1} \cL(t)$.
\label{fig:cst}}
\end{figure}

%

Our simulations are in accordance with the fact that $\frac1t \cL(t)$ converges a.s.\ to some constant, whose dependence on $\gamma$ is presented in Figure~\ref{fig:cst} (we present the result of only one simulation, but several simulations give values for $\frac1t \cL(t)$ very close to those presented here).  In view of the dependence on $\gamma$ of the constant $c_1$ in Theorem~\ref{thm:HLPP} (see in particular \eqref{formulaeVol}), a wild guess is that  the constant is proportional to $(1+\gamma)^{-1}\Gamma(1+\gamma)^{1/(1+\gamma)}$: the dotted grey line in Figure~\ref{fig:cst} represents the function $\gamma \mapsto  \frac{2^{3/2}}{1+\gamma}\Gamma(1+\gamma)^{1/(1+\gamma)}$ ---the factor $2^{3/2}$ is chosen so that it fits the value $\sqrt{2}$ when $\gamma=1$, corresponding to the standard Lipschitz LPP (the missing factor $\sqrt{2}$ comes from the length of the diagonal in Hammersley's LPP process). The two curves match quite closely, but they seem to disagree when $\gamma=0$ (the constant $\mathtt{C}_{1,1}$ seems very close to $2.75$, whereas $2^{3/2}\approx 2.83$).


\subsubsection{{\rm (2)} Convergence of the recentered and renormalized LPP}

In order to test the convergence in distribution of $ t^{-1/3}(\cL(t) - \mathtt{C}_{1,1} t)$, we performed $1000$ simulations for the point-to-point H$_{\gamma}$-LPP with $t=500$, for the three values $\gamma=0$, $\gamma=0.5$ and $\gamma=1.5$. 

\begin{figure}[htbp]
  \captionsetup{width=0.9\linewidth}
\captionsetup[subfigure]{width=0.8\linewidth}
\begin{center}
\begin{tabular}{ccc}
  \includegraphics[width=0.22\linewidth]{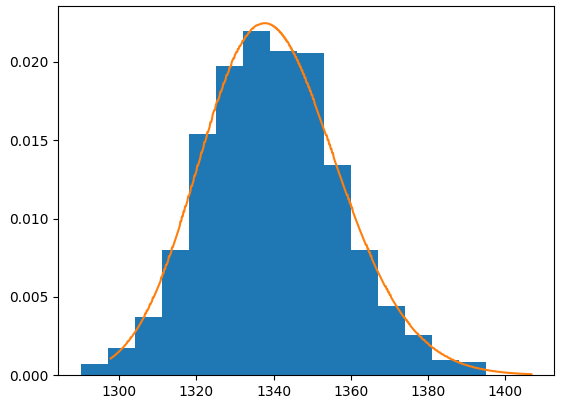}   
&\quad
\includegraphics[width=0.22\linewidth]{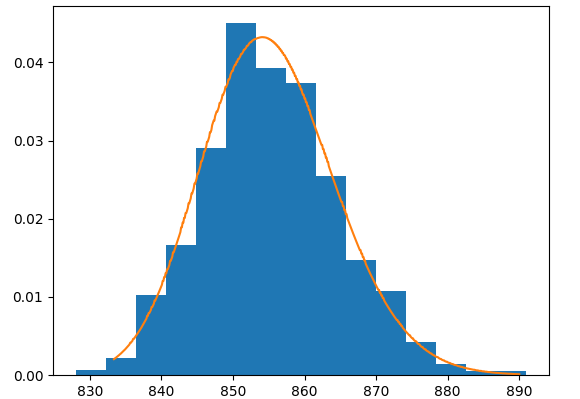}  
\quad 
&
 \includegraphics[width=0.22\linewidth]{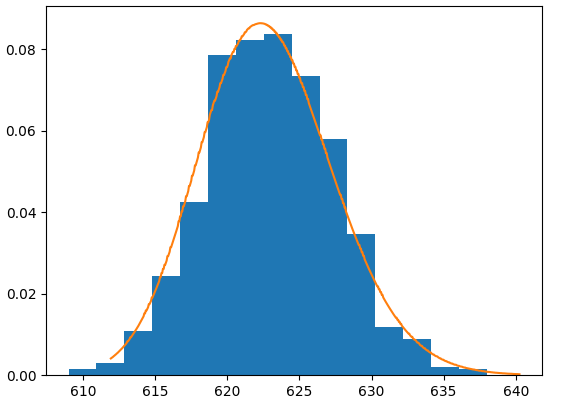} \\
\footnotesize (a) $\gamma=0$. & \footnotesize (b) $\gamma=0.5$. & \footnotesize (c) $\gamma=2$.
\end{tabular}
\vspace{-0.3cm}
    \caption{ \footnotesize Histograms of $k=10^3$ simulations of the point-to-point H$_{\gamma}$-LPP with $t=500$. The three subfigures (a), (b) and (c) correspond to the cases $\gamma=0$, $\gamma=1/2$ and $\gamma=3/2$ respectively. In each case, we superimpose the graph of the Tracy-Widom GUE density, after a recentering by $\mathtt{C}_{\gamma} t$ (with $\mathtt{C}_{\gamma} \approx 2.75 , 1.75, 1.26$ from left to right), and a renormalization by $c_{\gamma} t^{1/3}$ (with $c_{\gamma} \approx 2.5,1.3, 0.65$ from left to right).}
    \label{fig:histoHLPP}
  \end{center}
\end{figure}


The histograms presented in Figure~\ref{fig:histoHLPP} seem to confirm the convergence to a Tracy-Widom distribution, leading to a (far-reaching) conjecture, for the (point-to-point) H$_{\gamma}$-LPP.
\begin{conj}
\label{conj1}
For every $\gamma\ge 0$, there exists a constant $\mathtt{C}_{\gamma}$ (equal to $\frac{2^{3/2}}{1+\gamma}\Gamma(1+\gamma)^{1/(1+\gamma)}$?) and a constant $c_{\gamma}$ such that, for the point-to-point H$_{\gamma}$-LPP in Poisson environment with intensity $\lambda=1$ and $\gamma$-H\"older constraint $A=1$, we have
\begin{equation}
\frac{\cL(t) - \mathtt{C}_{\gamma}\, t }{ c_{\gamma}\, t^{1/3} } \xrightarrow[]{(\dd)}  F_{GUE} \quad \text{as } t\to+\infty \, .
\end{equation}
\end{conj}

\subsection{Directed E-LPP}
\label{app2}

As far as the directed E-LPP is concerned, we also performed simulations in the setting of Section~\ref{sec:poisson} with $t=100$, with a Poisson intensity $\lambda=1$ and a constraint $B=1$. Simulations are much less efficient, and the simulated annealing procedure only gives an approximate (under-estimated) value for $\cL(t)=\cL_\lambda^{(\cE_{a,b}^B)}(t,0) $.

\begin{figure}[htbp]
  \captionsetup{width=0.9\linewidth}
\begin{center}
\minipage{0.24\textwidth}
\includegraphics[width=\linewidth]{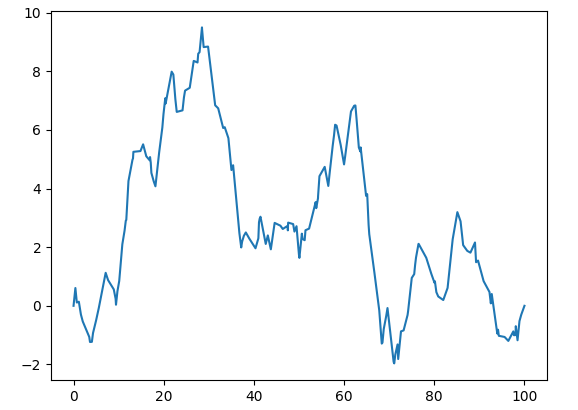}
\endminipage
\minipage{0.24\textwidth}
  \includegraphics[width=\linewidth]{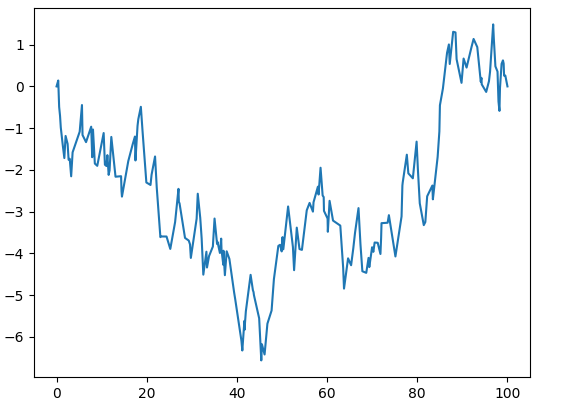}
\endminipage
\minipage{0.24\textwidth}
  \includegraphics[width=\linewidth]{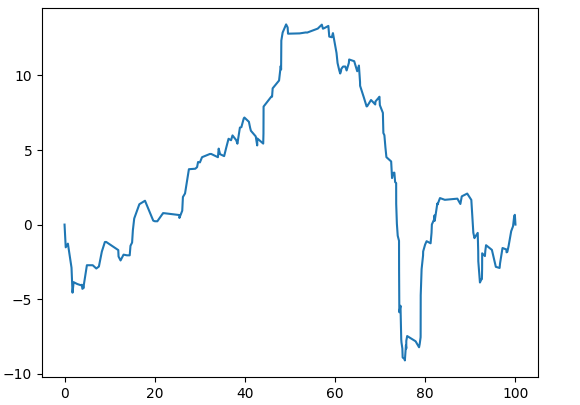}
\endminipage
\minipage{0.24\textwidth}
  \includegraphics[width=\linewidth]{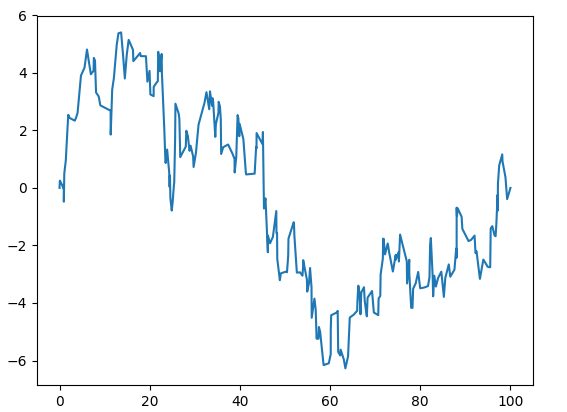}
\endminipage
\end{center}
\vspace{-0.3cm}
  \caption{\footnotesize Simulation of Poisson point-to-point E-LPP with $t=100$, via a simulated annealing procedure. The plots represents a path which collects a number of points that  approximate $\cL(t)$, with different parameters $a,b$. From left to right we have: $a=2,b=1$ ($\mathtt{C}_{1,1}\approx 1.83$), $a=4,b=1$ ($\mathtt{C}_{1,1}\approx 1.96$), $a=1,b=0$ ($\mathtt{C}_{1,1}\approx 2.08$), $a=2,b=0$ ($\mathtt{C}_{1,1}\approx 2.55$).  }
  \label{fig:cstELPP}
\end{figure}

Figure \ref{fig:cstELPP} presents some simulations to test the dependence of the constant $\mathtt{C}_{1,1} =\lim_{t\to\infty} \frac1t \cL(t)$ on the parameters $a,b$. The only conjecture we may risk to formulate (thanks to simulations for others values of $a,b$ that we do not present here) is that the constant should be non-decreasing in $a$ and non-increasing in $b$. Further conclusions are hard to draw from our simulations.

\begin{figure}[htbp]
  \captionsetup{width=0.9\linewidth}
  \begin{center}
\includegraphics[width=0.28\linewidth]{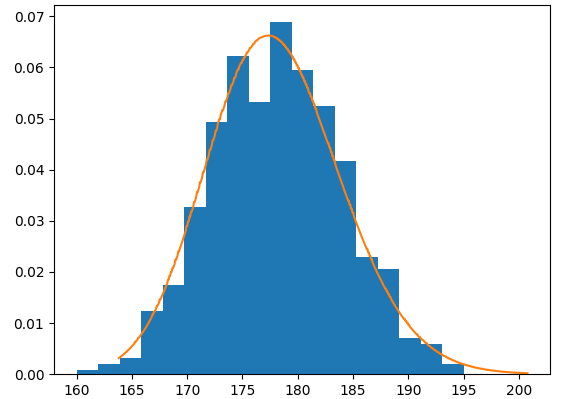}
\end{center}
\vspace{-0.3cm}
\caption{\footnotesize Histogram of 1000 realizations of $\cL(t)$ for $t=100$ (with $\lambda=1$, $B=1$), with $a=2,b=1$. We also plotted the graph of the GUE density, centered by $\mathtt{C_{a,b}} t$ with  $\mathtt{C}_{a,b}\approx 1.89$, and rescaled by $c_{a,b}\, t^{1/3}$ with $c_{a,b}\approx1.4$.
}
\label{fig:histoELPP}
\end{figure}

The histogram presented in Figure~\ref{fig:histoELPP} makes it natural  to conjecture that the Poisson point-to-point E-LPP, when properly recentered and renormalized, converges in distribution to a Tracy-Widom GUE distribution.

}

\end{appendix}

\bibliographystyle{plain}
\bibliography{biblioHTBN.bib}

\end{document}